\documentclass[12pt]{amsart}
\usepackage{amscd,amsmath,amsthm,amssymb}
\usepackage[left]{lineno}
\usepackage{pstcol,pst-plot,pst-3d}
\usepackage{color}
\usepackage{pstricks}
\usepackage{stmaryrd}
\usepackage[utf8]{inputenc}
\usepackage{pstricks-add}
\usepackage{epstopdf}
\usepackage{graphicx}
\usepackage{color}

\usepackage{amscd,amsmath,amsthm,amssymb,graphics}
\usepackage{slashbox}
\usepackage{pgf,tikz, pifont}
\usetikzlibrary{arrows}
\usepackage{lineno}
\usepackage{bm}

\newpsstyle{fatline}{linewidth=1.5pt}
\newpsstyle{fyp}{fillstyle=solid,fillcolor=verylight}
\definecolor{verylight}{gray}{0.97}
\definecolor{light}{gray}{0.9}
\definecolor{medium}{gray}{0.85}
\definecolor{dark}{gray}{0.6}

 \def\NZQ{\mathbb}               

 \def\FF{{\NZQ F}}

 \def\frk{\mathfrak}               

 \def\mm{{\frk m}}

 \def\G{{\mathcal G}}

 \def\Soc{{\mathbf Soc}}

 \def\opn#1#2{\def#1{\operatorname{#2}}} 
 \opn\chara{char} \opn\length{\ell} \opn\pd{pd} \opn\rk{rk}
 \opn\projdim{proj\,dim} \opn\injdim{inj\,dim} \opn\rank{rank}
 \opn\depth{depth} \opn\grade{grade} \opn\height{height}
 \opn\embdim{emb\,dim} \opn\codim{codim}
 
 \opn\Tr{Tr} \opn\bigrank{big\,rank}
 \opn\superheight{superheight}\opn\lcm{lcm}
 \opn\trdeg{tr\,deg}
 \opn\reg{reg} \opn\lreg{lreg} \opn\ini{in} \opn\lpd{lpd}
 \opn\size{size} \opn\sdepth{sdepth}
 \opn\link{link}\opn\fdepth{fdepth}\opn\lex{lex}
 \opn\tr{tr}
 \opn\type{type}
 \opn\gap{gap}
 \opn\arithdeg{arith-deg}
 \opn\Deg{Deg}
 \opn\sat{sat}
\opn\dist{dist}
\opn\soc{soc}
 \opn\div{div} \opn\Div{Div} \opn\cl{cl} \opn\Cl{Cl}
 \opn\Spec{Spec} \opn\Supp{Supp} \opn\supp{supp} \opn\Sing{Sing}
 \opn\Ass{Ass} \opn\Min{Min}\opn\Mon{Mon}
 \opn\Ann{Ann} \opn\Rad{Rad} \opn\Soc{Soc}
 \opn\Im{Im} \opn\Ker{Ker} \opn\Coker{Coker} \opn\Am{Am}
 \opn\Hom{Hom} \opn\Tor{Tor} \opn\Ext{Ext} \opn\End{End}
 \opn\Aut{Aut} \opn\id{id}
 
 \opn\nat{nat}
 \opn\pff{pf}
 \opn\Pf{Pf} \opn\GL{GL} \opn\SL{SL} \opn\mod{mod} \opn\ord{ord}
 \opn\Gin{Gin} \opn\Hilb{Hilb}\opn\sort{sort}
 \opn\PF{PF}\opn\Ap{Ap}
 \opn\mult{mult}
 \opn\bight{bight}
 \opn\aff{aff}
 \opn\relint{relint} \opn\st{st}
 \opn\lk{lk} \opn\cn{cn} \opn\core{core} \opn\vol{vol}  \opn\inp{inp} \opn\nilpot{nilpot}
 \opn\link{link} \opn\star{star}\opn\lex{lex}\opn\set{set}
 \opn\width{wd}
 \opn\Fr{F}
 \opn\QF{QF}
 \opn\G{G}
 \opn\type{type}\opn\res{res}
 \opn\conv{conv}
 \opn\Shad{Shad}
 \opn\gr{gr}
 
 \def\pot#1#2{#1[\kern-0.28ex[#2]\kern-0.28ex]}

 \opn\dirlim{\underrightarrow{\lim}}
 \opn\inivlim{\underleftarrow{\lim}}
 \let\sect=\cap
 
 \let\tensor=\otimes
 \let\iso=\cong
 \let\Union=\bigcup
 
 \let\Dirsum=\bigoplus
 
 \let\to=\rightarrow
 
 \def\Implies{\ifmmode\Longrightarrow \else
         \unskip${}\Longrightarrow{}$\ignorespaces\fi}
 \def\implies{\ifmmode\Rightarrow \else
         \unskip${}\Rightarrow{}$\ignorespaces\fi}
 \def\iff{\ifmmode\Longleftrightarrow \else
         \unskip${}\Longleftrightarrow{}$\ignorespaces\fi}

 \let\:=\colon
 \newtheorem{Theorem}{Theorem}[section]
 \newtheorem{Lemma}[Theorem]{Lemma}
 \newtheorem{Corollary}[Theorem]{Corollary}
 \newtheorem{Proposition}[Theorem]{Proposition}
 \newtheorem{Remark}[Theorem]{Remark}
 
 \newtheorem{Example}[Theorem]{Example}

 \let\epsilon\varepsilon
 \let\kappa=\varkappa
 \def\qed{\ifhmode\textqed\fi
       \ifmmode\ifinner\quad\qedsymbol\else\dispqed\fi\fi}
 \def\textqed{\unskip\nobreak\penalty50
        \hskip2em\hbox{}\nobreak\hfil\qedsymbol
        \parfillskip=0pt \finalhyphendemerits=0}
 \def\dispqed{\rlap{\qquad\qedsymbol}}
 
 \opn\dis{dis}
 \def\pnt{{\raise0.5mm\hbox{\large\bf.}}}
 
 \opn\Lex{Lex}

\begin{document}

\title {The socle module of a monomial ideal}

\author{Lizhong Chu, J\"urgen Herzog, Dancheng Lu$^*$}
\address{Lizhong Chu, School of  Mathematical Sciences, Soochow University, P.R.China}
\email{chulizhong@suda.edu.cn}

\address{J\"urgen Herzog, Fachbereich Mathematik, Universit\"at Duisburg-Essen, Campus Essen, 45117
Essen, Germany} \email{juergen.herzog@uni-essen.de}

\address{Dancheng Lu, School of  Mathematical Sciences, Soochow University, P.R.China}
\email{ludancheng@suda.edu.cn}

\dedicatory{ }

\begin{abstract}
For any ideal $I$ in a Noetherian local ring  or any graded ideal $I$ in a standard graded $K$-algebra over a field $K$,  we introduce the  socle module $\Soc(I)$,  whose graded components give us the socle of the powers of $I$. It is observed that $\Soc(I)$  is a  finitely generated module over the fiber cone of $I$. In the case that $S$ is the polynomial ring and all powers of $I\subseteq S$ have linear resolution,  we  define  the module $\Soc^*(I)$ which is a module over the Rees ring of $I$.  For the edge ideal of a graph and for classes of polymatroidal ideals we study the module structure of their socle modules.
\end{abstract}

\thanks{$^*$ Corresponding author}





\maketitle

\setcounter{tocdepth}{1}


\section*{Introduction}

Let $(S, \mm)$ be a Noetherian local ring $S$ with maximal ideal $\mm$  or a standard graded $K$-algebra $S$ over a field $K$  with  maximal graded ideal $\mm$. Furthermore, let $I\subseteq \mm$ be an ideal which we assume to be graded if $S$ is standard graded. In this paper we study for each integer $m\geq 0$  the socle $(I^m:\mm)/I^m$ of the powers $I^m$ of $I$. Results in this direction can be found in \cite{H-H}, \cite{H-Q} and in \cite{L-T}.

It is clear that the multiplication with any  $f\in I^r$ induces a map $(I^m:\mm)/I^m\to (I^{m+r}:\mm)/I^{m+r}$, so that $\Soc(I)=\Dirsum_{m\geq 0} (I^m:\mm)/I^m$ has the structure of a graded $\mathcal{F}(I)$-module, where $\mathcal{F}(I)=  \Dirsum_{m\geq 0}I^m/\mm I^m$ is the fiber cone of $I$. In Proposition~\ref{finite} we notice that $\Soc(I)$ is a finitely generated $\mathcal{F}(I)$-module. The proof is based on the fact that,  up to a truncation, $\Soc(I)$ can be identified with an ideal of the associated graded ring $\gr_I(S)$ of $I$.

In this paper we are mostly interested in the module structure of $\Soc(I)$, when $I$ is a monomial ideal and in particular in the case when $I$ is the edge ideal of a graph or a polymatroidal ideal. In both of these cases, $I$ satisfies the Ratliff condition, which means that $(I^{m+1}:I)=I^m$ for all $m\geq 0$, see \cite[Lemma 2.12]{M-M-V} and \cite[Theorem 4.1]{H-Q}. In this case, $\Soc(I)$, without any truncation, may be viewed as an ideal of $\gr_I(S)$. As a consequence of Theorem~\ref{sum}, this fact can be used to show that if $G$ is a finite simple graph with connected components $G_1,\ldots,G_r$, then $\Soc(I(G))=\Soc(I(G_1))\Soc(I(G_2))\cdots \Soc(I(G_r))$. Here,  $I(H)$ denotes the edge ideal of a graph $H$.

If $I$ has a $d$-linear resolution, then, as observed in Proposition~\ref{d-1},  $(I:\mm) =I+\soc(I)$, where $\soc(I)$  is an ideal generated in degree $d-1$. Hence if all powers of $I$ have linear resolution, which is the case for polymatroidal ideals, we may define $\Soc^*(I)=\Dirsum_{m\geq 1}\soc(I^m)$, which is a graded module over the Rees ring of $I$. Note that $\Soc(I)\iso \Soc^*(I)/\mm \Soc^*(I)$.

The main result of Section~2 is Theorem~\ref{fiber}, where it is shown that if  $G$ is a unicyclic graph with the  unique odd cycle $C_{2k+1}$ and with edge ideal $I=I(G)$,  then $\Soc(I)\iso  \mathcal{F}(I)(-c-k)$. Here, $c$ is the number of edges which do not belong to the cycle $C_{2k+1}$ and which are not leaves. The proof of this theorem uses in a crucial way Lemma~\ref{connectedness} which follows from Theorem~6.1 and
Theorem~6.7 of Banerjee in \cite{B}. The generator of $\Soc(I)$ is of the form $u+I^{c+k+1}$,  where $u$ is a monomial of degree $2(c+k)+1$. It is clear that for any graph the generators of $\Soc(I)$ are of the form $u+I^m$ where $u$ is a monomial of degree $\geq 2m-1$. If the graph $G$ has more than one odd cycle, then $\Soc(I)$ may have generators $u+I^m$ with $\deg(u)>2m-1$, see Example~\ref{higher}. On the other hand, it is shown in Proposition~\ref{two} that if any two odd cycles of $G$ have distance at most one, then each monomial  $u\in (I^m:\mm)\setminus  I^m$
has degree $2m-1$.  Section 2 closes with an application of our results to the depth stability index of  the edge ideal of a connected nonbipartite graph, see Corollary~\ref{dstab}.

Section~3 is devoted to the study of $\Soc^*(I)$ when $I\subseteq K[x_1,\ldots,x_n]$ is a polymatroidal ideal. It follows from \cite[Corollary 3.5]{H-R-V} that $\Soc^*(I)\neq 0$ if and only if the analytic spread $\ell(I)$ of $I$ is $n$, and from \cite[Theorem 4.1]{H-Q}
that the least degree of a generator of $\Soc^*(I)$ is $< n$. We have some computational evidence that a much stronger statement holds, namely that  $\Soc^*(I)$ is generated in degree $<n$. The main result of Section~3 is Theorem~\ref{mainPLP}. There it is shown that indeed $\Soc^*(I)$ is generated in degree $< n$ for a PLP-polymatroidal ideal $I$ on the ground set $[n]$. The pruned lattice path polymatroids (PLP-polymatroidals for short)  were firstly introduced by J. Schweig in \cite{S}. Then the third author of  this paper in \cite{L} observed that this class of polymatroidal ideals   is  given  by a system of linear inequalities and they constitute a wide generalization of polymatroidal ideals with strong exchange properties whose polymatroidal ideals are essentially of Veronese type.  In Proposition~\ref{single} those Veronese type ideals are characterized for which $\Soc^*(I)$ is equi-generated. For the proof of these results one has to analyze carefully the system of linear inequalities defining  this type of polymatroids.

\section{Definition and basis properties of the socle module}

We define the {\em socle module} of $I$ to be
\[
\Soc(I)=\Dirsum_{m\geq 0}(I^m:\mm)/I^m.
\]
Note that $\Soc(I)$ is a graded  $\mathcal{F}(I)$-module, where  $\mathcal{F}(I)  =\Dirsum_{m\geq 0}I^m/\mm I^m$ is the fiber cone  of $I$.

\begin{Proposition}
\label{finite}
The module  $\Soc(I)$ is a finitely generated graded $\mathcal{F}(I)$-module.
\end{Proposition}

\begin{proof}
We denote by $\gr_I(S)=\Dirsum_{m\geq 0}I^m/I^{m+1}$ the associated graded ring of $I$. Then the module $C(I)=(0:_{\gr_I(S)}\mm)=\{f\in \gr_I(S)\:\; \mm f=0\}$ is a finitely generated $\gr_I(S)$-module because it is an ideal in  the Noetherian ring $\gr_I(S)$.  Since $(0:_{\gr_I(S)}\mm)$ is annihilated by $\mm$ and since $\mathcal{F}(I) =\gr_I(S)/\mm \gr_I(S)$ it is actually a finitely generated $\mathcal{F}(I)$-module.

The $m$th graded component of $C(I)$ is equal to $((I^m:\mm)\sect I^{m-1})/I^m$. Ratliff \cite{R} showed that there exists $m_0$ such that  $(I^{m}:I)=I^{m-1}$ for all $m\geq m_0$. Since  $(I^{m} :\mm)\subseteq (I^{m} :I)$,  we see that $\Soc(I)_{\geq m_0} =C(I)_{\geq m_0}$, where for any graded module $M$,  $M_{\geq r}$ denotes the submodule  $\Dirsum_{m\geq r}M_m$ of $M$.  This shows that $\Soc(I)_{\geq m_0}$ is finitely generated, and consequently $\Soc(I)$ is finitely generated as well.
\end{proof}

We say that $I$ satisfies the {\em  Ratliff condition} if $(I^{m+1}:I)=I^m$  for all $m\geq 0$. Let $C(I)=(0:_{\gr_I(S)}\mm)$ be the module introduced in the proof of Proposition~\ref{finite}.  Note that $C(I)=\Soc(I)$, if $I$ satisfies the Ratliff condition.

\medskip
Let $I_1\subseteq S_1=K[x_1,\ldots,x_n]$ and $I_2\subseteq S_2=K[y_1,\ldots, y_m]$ be monomial ideals, and let $I=(I_1, I_2)S\subseteq S$, where $S=S_1\tensor _KS_2$. By \cite[Proposition 3.2]{H-T-T}, the canonical map $\alpha\:  \gr_{I_1}(S_1)\tensor_K \gr_{I_2}(S_2)\to \gr_{I}(S)$ with $f\tensor g\mapsto fg$ is an isomorphism. Here,  since $\gr_{I_1}(S_1)\to \gr_{I}(S)$ is injective, we identify $f\in \gr_{I_1}(S_1)$  with its image in $\gr_{I}(S)$. A similar identification is made for $g\in \gr_{I_2}(S_2)$.

For an ideal $J\subseteq \gr_{I_i}(S_i)$ we denote its extension ideal  under the canonical (injective) map $\gr_{I_i}(S_i)\to \gr_I(S)$ again by $J$. With this convention we have

\begin{Theorem}
\label{sum}
Let $I_1\subseteq S_1=K[x_1,\ldots,x_m]$ and $I_2\subseteq S_2=K[y_1,\ldots, y_n]$ be monomial ideals, and let $I=(I_1,I_2)S\subseteq S$, where $S=S_1\tensor _KS_2$. Assume that $I_1$, $I_2$ and $I$ satisfy the Ratliff condition. Then $\Soc(I)=\Soc(I_1)\Soc(I_2)$.
\end{Theorem}

\begin{proof}
Let $A_i=\gr_{i_1}(S_i)$ for $i=1,2$. Our assumptions imply that $\Soc(I_1)$ is an ideal in $A_1$, $\Soc(I_2)$ an ideal in $A_2$ and $\Soc(I)$ an ideal in $A=\gr_I(S)$. Then $\Soc(I_1)\tensor_K\Soc(I_2)$ is an ideal in $A_1\tensor_K A_2$.  We show that $\alpha(\Soc(I_1)\tensor_K\Soc(I_2))=\Soc(I)$.

The inclusion $\alpha(\Soc(I_1)\tensor_K\Soc(I_2))\subseteq\Soc(I)$ is obvious.
For the converse inclusion, we first notice that $\Soc(I)$ is generated by monomials. Indeed, the elements  of degree $k$ in  $\Soc(I)$ which are not zero are of the  form $u+ I^{k+1}\in I^k/I^{k+1}$ with $u\in (I^{k+1}:\mm)\setminus  I^{k+1}$, where $\mm =(x_1,\ldots,x_m,y_1,\ldots,y_m)$. Since $I$ is a monomial ideal, $I^{k+1}:\mm$ is a monomial ideal as well, and the assertion follows.

Let $u+I^{k+1}\in \Soc(I)_{k}$ with $u$ a monomial. Since $\alpha$ is a multigraded isomorphism there exists monomials $u_i\in S_i$ and  non-negative positive integers $r$ and $s$ with $r+s=k$ such that $\alpha((u_1+I_1^{r+1})\tensor (u_2+I_2^{s+1}))=u+I^{k+1}$. Since $x_i(u+I^k)=0$ and since $\alpha$ is an isomorphism, it follows that
\[
(x_iu_1+I_1^{r+1})\tensor (u_2+I_2^{s+1})=(x_i\tensor 1)((u_1+I_1^{r+1})\tensor (u_2+I_2^{s+1}))=0.
\]
Note that $u_2+I_2^{s+1}\neq 0$, because otherwise $(u_1+I_1^{r+1})\tensor (u_2+I_2^{s+1})=0$, and hence $u+ I^{k+1}=0$. On the other hand, since
$(x_iu_1+I^{r+1})\tensor (u_2+I^{s+1})=0$,  it follows  that $(x_iu_1+I_1^{r+1})=0$. This is the case for all $i$, which implies that $u_1+I_1^{r+1}\in \Soc(I_1)$. Similarly we see that $u_2+I_2^{s+1}\in \Soc(I_2)$.
\end{proof}

Let $I\subseteq S$ be a monomial ideal generated in degree $d$, and $u\in S$ a monomial  such that $u+I^m\in \Soc(I)$ and $u+I^m\neq 0$. What can we say about the degree $\deg(u)$ of $u$? Obviously, $\deg(u)\geq md-1$. If $\deg(u)=md-1$ for all such elements $u$, then the socle module $\Soc(I)$ has the following good property.

\begin{Proposition} \label{rank} Let $I\subseteq S$ be a monomial ideal generated  in degree $d$. Suppose that for any positive integer $m$ and for any monomial $u\in (I^m:\mm)\setminus I^m$, one has $\deg(u)=md-1$. Then $\Soc (I)$ is identified with an ideal of $\mathcal{F}(I)$. In particular, $\rank \Soc(I)=1$ if\quad $\Soc(I)\neq 0$.
\end{Proposition}
\begin{proof} Fix a variable $x\in S$. We define $\varphi$ to be the map $\varphi: \Soc(I)\longrightarrow \mathcal{F}(I)$ with the property   that $\varphi(u+I^m)=xu+\mm I^m$ for any $u+I^m\in (I^m:\mm)/ I^m$ and that $\varphi(u_1+u_2)=\varphi(u_1)+\varphi(u_2)$ for any $u_1,u_2\in \Soc(I).$

It is clear that $\varphi$ is well-defined. Note that if a monomial $u\in (I^m:\mm)\setminus  I^m$ then $\deg (xu)=md$ and $xu\notin \mm I^m$, we have $\varphi$ is injective. It remains to show that $\varphi$ is a $\mathcal{F}(I)$-module homomorphism. Let $u+I^m\in (I^m:\mm)/ I^m$ and $v+\mm I^r\in I^r/\mm I^r$. Then
\[\varphi((v+\mm I^r)(u+I^m))=\varphi(vu+I^{m+r})=xvu+\mm I^{m+r}=(v+\mm I^r)\varphi(u+I^m).
\]
By this equality and since $\varphi$ preserves the addition, it follows that  $\varphi$ is indeed a $\mathcal{F}(I)$-module homomorphism, as required.

Finally, we note that  $\mathcal{F}(I)$ is an integral domain (in fact, it is a subalgebra of the polynomial ring $S$) since $I$ is generated in a single degree. Now, the last result follows since every nonzero ideal of an integral domain has rank 1.
\end{proof}

 Unfortunately, the inequality $\deg(u)\geq md-1$, where $u$ is a homogeneous  element in $(I^m:\mm)\setminus I^m$, can be strict. Indeed, consider the  edge ideal $I(G)$  of the graph $G$ consisting of two disjoint $3$-cycles which are connected by a path of length $2$, see $G_1$ in Figure~\ref{G}.  Then
\[
I(G)=(x_1x_2,x_1x_3,x_2x_3,y_1y_2,y_2y_3,y_2y_3,x_1z,y_1z).
\]
It can be easily seen that $u\in (I(G)^3:\mm)\setminus I(G)^3$,  where  $u=x_1x_2x_3y_1y_2y_3$. Here,  $\deg(u)>3\cdot 2-1$.

\medskip
On the other hand, we have

\begin{Proposition}
\label{d-1}
Suppose that the graded ideal $I$ in the polynomial ring $S$ has $d$-linear resolution. Then each homogeneous element in $(I:\mm) \setminus I$ is of degree $d-1$
\end{Proposition}

\begin{proof}
Let
\[
\FF: 0\to F_{n-1}\to \cdots \to F_1\to F_0\to I\to 0.
\]
be the graded free $S$-resolution of $I$. Since $I$ has $d$-linear resolution, it follows that $F_{n-1}=S(-d-(n-1))^{\beta_{n-1}}$.

We have  the following isomorphisms of graded modules
\[
K(-d-(n-1))^{\beta_{n-1}}\iso \Tor_{n-1}(K, I)\iso \Tor_n(K, S/I)\iso H_n(x_1,\ldots,x_n;S/I).
\]
Here $H_n(x_1,\ldots,x_n;S/I)$ denotes the $n$th Koszul homology of $S/I$ with respect to the sequence $x_1,\ldots,x_n$.
Note that $H_n(x_1,\ldots,x_n;S/I)=((I:\mm) /I)\bigwedge^n E$, where $E=\Dirsum_{i=1}^n(S/I)e_i$.  Hence $((I:\mm)/I)\bigwedge^n E=((I:\mm)/I)(e_1\wedge e_2\wedge \cdots \wedge e_n)=K(-d-(n-1))^{\beta_{n-1}}$.
Comparing degrees we obtain the desired conclusion, since $\deg(e_1\wedge e_2\wedge \cdots \wedge e_n)=n$.
\end{proof}

Proposition~\ref{d-1} implies that if $I$ is a graded ideal generated in degree $d$ which has the property that all its powers have linear resolution, then $\Soc(I)$ is generated by elements of the form $f+I^k$ where $f$ is homogeneous of  degree $kd-1$. Classes of such ideals are the polymatroidal ideals which we consider in the last section of this paper. In Proposition~\ref{two}, we provide a large  class of   edge ideals $I$ which may have not linear resolutions, but the equality $\deg(u)= md-1=2m-1$ still holds for any monomial $u\in (I^m:\mm)\setminus I^m$.
\medskip

\medskip
There is a natural condition on monomial ideals $I$ generated in degree $d$ which guarantees the existence of at least one generator $u+I^k$ of $\Soc(I)$ with $\deg u=kd-1$.

Let $G(I)=\{u_1,\ldots,u_m\}$. The {\em linear relation graph} $\Gamma$ of $I$ is the graph with edge set
\[
E(\Gamma)=\{\{i,j\}\: \text{there exist $u_k,u_l\in G(I)$ such that $x_iu_k=x_ju_l$}\}
\]
and vertex set $V(\Gamma)=\Union_{\{i,j\}\in E(\Gamma)}\{i,j\}$.

In the special case that $I=I(G)$ is the edge  ideal of a simple graph $G$ on the vertex set $[n]$,
the linear relation graph $\Gamma$ of $I(G)$ has edge set
\begin{eqnarray*}
\label{easy}
\{ \{i,j\} \; : \text{$i, j \in V(G)$ and $i$ and $j$ have a common neighbor in $G$}\}.
\end{eqnarray*}

The next result  follows from \cite[Corollary~3.4]{H-Q}.

\begin{Proposition}
\label{relation}
Let $I\subseteq S$ be a monomial ideal generated in degree $d$ with linear relation graph $\Gamma$.  Suppose that $\Gamma$ has $n$ vertices and that $\Gamma$ is connected. Then there exists a monomial $u\in (I^n:\mm)\setminus I^n$ of degree $dn-1$. In particular, $\Soc(I)$ admits a generator $u+I^k$  with $k\leq n$ and $\deg u=kd-1$.
\end{Proposition}

\section{Edge Ideals}

In this section, we always assume that  $G$ is  a simple graph on vertex set $V(G)=\{x_1, \ldots, x_n\}$ unless otherwise stated. The set of edges of $G$ is denoted by $E(G)$.  Let $S=K[x_1,\ldots,x_n]$ be  the polynomial ring
 over a field $K$. The {\it edge ideal} of $G$,
denoted by $I(G)$, is the ideal of $S$ generated by all
square-free monomials $x_ix_j$ such that $\{x_i,x_j\}\in E(G)$. We
often write $x_ix_j\in E(G)$ instead of $\{x_i,x_j\}\in E(G)$. By
abusing notation, we use  $x_ix_j$ to refer to both the edge
$x_ix_j\in E(G)$ and the monomial $x_ix_j\in I(G)$. We will study the socle module of an edge ideal of $G$. In view of Theorem~\ref{sum}, we  always assume that $G$ is a  connected graph.

Recall some necessary graphic notions.  A {\it walk} in $G$ is an alternating sequence of vertices and edges
$x_1, e_1, x_2, e_2,\ldots, e_{t-1}, x_t$ in which $e_i=x_i
x_{i+1}$. The number $t-1$ is the {\it length} of this walk. We often write this walk simply  as:  $x_1-x_2-\cdots-x_{t}.$  If $x_1=x_t$, it is called a {\it closed walk}. A {\it path}  is a simple walk, that is, a walk in which
each vertex appears exactly once. A {\it cycle}  is a closed walk in which each vertex appears exactly once except  the first and
the last vertices. A cycle of odd length is called an {\it odd cycle}.


We use $\{\{\ldots\}\}$ to refer to a {\it multiset}. For examples, $\{\{1,2\}\}\neq \{\{1,2,2\}\}$, $\{\{1,2,2\}\}\subseteq \{\{1,2,2,3,4\}\}$, but $\{\{1,2,2\}\}\nsubseteq \{\{1,2,3,4\}\}$, and $\{\{1,2,2,3,3\}\}\cap \{\{2,2,2, 3\}\}=\{\{2,2,3\}\}$.
The next  lemma, which is a variation of  \cite[Theorem~6.1 and
Theorem~6.7]{B}, plays a key role in this section.

\begin{Lemma}
\label{connectedness}   Let $G$ be a simple graph with edge ideal
$I=I(G)$. Let $m\geq 1$ be an integer and  let  $e_1, e_2, \ldots, e_m$ (maybe repeatedly) be edges  of $G$  and $\mathbf{v}\in S$ a monomial  such that
$e_1e_2\cdots e_m\mathbf{ v} \in I^{m+1}$.  Then there exist variables  $w$ and $y$
with $wy|\mathbf{v}$  and an odd walk in $G$ connecting $w$ to $y$:
$$ w=z_1-z_2-z_3-\cdots-z_{2t}-z_{2t+1}-z_{2t+2}=y$$
 such that $\{\{z_2z_3,\ldots,z_{2t}z_{2t+1}\}\}\subseteq  \{\{e_1,\ldots,e_m\}\}$. Here, if $t=0$, then the walk means the edge $w-y$ (i.e., the edge $wy$).
\end{Lemma}

Let $U, W$ be two subsets of $V(G)$. Then the distance between $U$ and $W$, denoted by $\dist(U,W)$, is defined as
$$\dist(U,W)=\min\{\dist(x,y) \:\; x\in U,  y\in W \},$$  where $\dist(x,y)$
denotes the minimal of lengths  of the paths between $x$ and $y$.
By convention,  $\dist(x,y)=0$ if $x=y$ and $\dist(x,y)=\infty$ if there is no path between $x$
and $y$. We observe the following easy fact: every odd closed walk (i.e., closed walk of odd length) contains an odd cycle. More precisely,
if $W$ is a closed odd walk of $G$, then $W$ has a subgraph $C$ such that $C$ is an odd cycle. Here, $C$ is a subgraph of $W$ means $V(C)\subseteq V(W)$ and $E(C)\subseteq E(W).$

\begin{Proposition}
 \label{two} Let $G$ be a simple connected graph such that $\dist(V(C_1),V( C_2))\leq 1$  for any two odd cycles $C_1,C_2$ of $G$. Then for any positive integer $m$ and for any  monomial  $\mathbf{u}$ with $\mathbf{u}\in (I^m:\mathfrak{m})\setminus I^m$, one has $\deg (\mathbf{u})=2m-1$ and $\mathbf{u}$ can be written as $\mathbf{u}=x_1x_2\cdots x_{2k+1}e_1\cdots e_{m-k-1}$, where $e_i$ is an edge for each $i$ and $x_1-x_2-\cdots  -x_{2k+1}-x_1$ is an odd cycle.
\end{Proposition}

\begin{proof} Let $\mathbf{u}$ be a monomial in $(I^m:\mathfrak{m})\setminus I^m$. Then $\mathbf{u}=e_1\cdots e_{m-1}\mathbf{v}$, where  $e_i\in E(G)$ for $i=1,\ldots, m-1$ and $\mathbf{v}$ is a monomial.  We need to show $\deg (\mathbf{v})=1$. Assume on the contrary that there are variables $x,y$ such that $xy$ divides $\mathbf{v}$. Since $x\mathbf{u}\in I^m$, it follows from Lemma~\ref{connectedness} that there is an integer $t\geq 1$ and  a walk $$z_1-z_2-z_3-\cdots-z_{2t}-z_{2t+1}-z_{2t+2}$$  in $G$
such that $z_1z_{2t+2}$ divides $x\mathbf{v}$ and $$\{\{z_2z_3, z_4z_5, \ldots, z_{2t}z_{2t+1} \}\}\subseteq\{\{e_1,\ldots,e_{m-1}\}\}.$$ Since $$z_1(z_2z_3)\cdots (z_{2t}z_{2t+1})z_{2t+2}=(z_1z_2)(z_3z_4)\cdots (z_{2t+1}z_{2t+2})\in I^{t+1},$$ we have  $$z_1z_{2t+2}e_1\cdots e_{m-1}\in I^m.$$ This implies that $z_1z_{2t+2}$ does not divide $\mathbf{v}$, since $\mathbf{u}\notin I^m$. Hence, we have  $$z_1z_{2t+2}|x\mathbf{v},  \qquad z_1z_{2t+2}\nmid \mathbf{v},\qquad  xy|\mathbf{v}.$$  From this it follows that   $z_1=z_{2t+2}=x$ and $x\neq y$.

 By the  argument above, we see that $x=z_1-z_2-z_3-\cdots-z_{2t}-z_{2t+1}-z_{2t+2}=x$ is an  odd closed walk. Hence, by the relabelling of vertices,
 there exist odd closed walks $C_1: x=x_1-x_2-x_3-\cdots-x_{2k}-x_{2k+1}-x_1=x$ and $C_2: y=y_1-y_2-\cdots -y_{2\ell}-y_{2\ell+1}-y_1=y$ such that
$$\{\{f_1,\ldots,f_{k}\}\}\subseteq \{\{e_1,\ldots,e_{m-1}\}\}$$ and $$\{\{g_1,\ldots,g_{\ell}\}\} \subseteq \{\{e_1,\ldots,e_{m-1}\}\},$$
where $f_i=x_{2i}x_{2i+1}$  for $i=1,\ldots,k$ and $g_i=y_{2i}y_{2i+1}$ for $i=1,\ldots,\ell$.

Denote by $F$ the multi-set $\{\{f_1,\ldots,f_{k}\}\}\cap \{\{g_1,\ldots,g_{\ell}\}\}$. We firstly  assume that  $F\neq \emptyset$. Then $$\mathbf{u}_1:=\frac{xyf_1\cdots f_k g_1\cdots g_{\ell}}{\prod_{e\in F}e} \mbox{\quad divides \quad} \mathbf{u},$$  and moreover, $$\frac{\mathbf{u}}{\mathbf{u}_1} \mbox{ is the product of some edges } .$$  This implies $$\mathbf{u}_1\in I^{k+\ell-|F|}\setminus I^{k+\ell-|F|+1}.$$ Let $s$ be the minimal of $i$ with $f_i\in F$ and let $1\leq t\leq \ell$ such that $f_s=g_t$. Then either $x_{2s}=y_{2t+1}$ and $x_{2s+1}=y_{2t}$ or $x_{2s}=y_{2t}$ and $x_{2s+1}=y_{2t+1}$. Suppose that $x_{2s}=y_{2t+1}$ and $x_{2s+1}=y_{2t}$. In this case, $\mathbf{u}_1$ could be written as
 $$\mathbf{u}_1=\prod_{i=0}^{s-1}( x_{2i+1}x_{2i+2})\cdot (x_{2s+1}y_{2t-1})\cdot \prod_{i=1}^{t-1}(y_{2i-1}y_{2i})\cdot  \prod_{i=t+1}^{\ell}(y_{2i}y_{2i+1})\cdot\frac{\prod_{i=s+1}^k f_i} {\prod_{e\in F\setminus \{\{f_s\}\}}e}.$$
By the choice of $s$, we see that $F\setminus \{\{f_s\}\} \subseteq \{\{f_{s+1},\cdots,f_k\}\}$. From this it follows that $\mathbf{u}_1$ is the product of edges and $\mathbf{u}_1\in I^{k+\ell-|F|+1}$. This is  a contradiction. The case that $x_{2s}=y_{2t}$ and $x_{2s}=y_{2t+1}$  yields a contradiction similarly.

Secondly, we assume that  $F=\emptyset$. Then $$\mathbf{u}_2:=xf_1\cdots f_{k}yg_1\cdots g_{\ell}=x_1x_2\cdots x_{2k+1}y_1y_2\cdots y_{2\ell+1}\in I^{k+\ell}\setminus I^{k+\ell+1}.$$
If $V(C_1)\cap V(C_2)=\emptyset$, there exist some vertex of $C_1$, say $x_1$, and some vertex of $C_2$, say $y_1$, such that $x_1y_1\in E(G)$,  since $\dist(V(C_1),V(C_2))\leq 1$ (Here, we use the observation that every odd closed walk contains an odd cycle). Thus,
$$\mathbf{u}_2=(x_1y_1) (x_2x_3) \cdots (x_{2k}x_{2k+1}) (y_2y_3) \cdots (y_{2k}y_{2k+1})\in I^{k+\ell+1},$$ a contradiction.
If $V(C_1)\cap V(C_2)\neq \emptyset$, we may harmlessly assume that $x_1=y_1$. Then
$$\mathbf{u}_2=(x_1y_{2\ell +1}) (x_2x_3) \cdots (x_{2k}x_{2k+1}) (y_1y_2) \cdots (y_{2k-1}y_{2k})\in I^{k+\ell+1}.$$
 This is also a contradiction.   The contradictions obtained above show that  $\deg(\mathbf{v})=1$ and  $\deg(\mathbf{u})=2m-1$.  Moreover, by the first paragraph of this proof, we may write  $\mathbf{u}$ as $\mathbf{u}=x_1x_2\cdots x_{2k+1}e_1\cdots e_{m-k-1}$, where $\{x_1,\ldots,x_{2k+1}\}$ are the vertex set of an odd cycle and  $e_i\in E(G)$ for $i=1,\cdots,m-k-1$. Here, we use the fact that if $z_1-z_2-\cdots-z_{2k+1}-z_1$ is an odd closed walk containing an odd cycle $x_1-x_2-\cdots-x_{2\ell+1}-x_1$ then $z_1\cdots z_{2k+1}$ is the product of $x_1\cdots x_{2\ell +1}$ and some edges.
  \end{proof}

  \begin{Example}
\label{higher}\em  Let $G_1,G_2$ be in Figure~\ref{G}, and set $I_i=I(G_i)$ for $i=1,2$. Then $x_1x_2x_3x_4x_5x_6\in (I_1^3:\mm)\setminus I_1^3$.  This illustrates that the requirement on the distance of odd cycles in Proposition~\ref{two} cannot be removed.

In contrast, if a monomial $\mathbf{u}\in (I_2^m:\mathfrak{m})\setminus I_2^m$, then $\deg (\mathbf{u})=2m-1.$ Thus,  the converse of Proposition~\ref{two} does not hold.
\begin{figure}[ht!]

\begin{tikzpicture}[line cap=round,line join=round,>=triangle 45,x=1.5cm,y=1.5cm]

\draw (4,1)--(4,3)--(5,2)--(5.7,2)--(6.4,2)--(7.4,3); \draw (6.4,2)--(7.4,1)--(7.4,3); \draw (4,1)--(5,2);

\draw (4.2,1) node[anchor=north east]{1};
\draw (5.2,1.8) node[anchor=north east]{3};
\draw (4.4,3.2) node[anchor=north east]{2};

\draw (5.8,1.8) node[anchor=north east]{7};
\draw (6.4,1.8) node[anchor=north east]{4};
\draw (7.4,3.2) node[anchor=north east]{5};
\draw (7.4,1) node[anchor=north east]{5};

\fill [color=black] (4,3) circle (1.5pt);
\fill [color=black] (4,1) circle (1.5pt);
\fill [color=black] (5,2) circle (1.5pt);

\fill [color=black] (5.7,2) circle (1.5pt);
\fill [color=black] (6.4,2) circle (1.5pt);
\fill [color=black] (7.4,3) circle (1.5pt);

\fill [color=black] (7.4,1) circle (1.5pt);

\draw (9,1)--(9,3)--(10,2)--(10.7,2)--(11.4,2)--(12.4,3); \draw (11.4,2)--(12.4,1)--(12.4,3); \draw (9,1)--(10,2);
\draw (10.7,2)--(10.7,3);
\fill [color=black] (9,3) circle (1.5pt);
\fill [color=black] (9,1) circle (1.5pt);
\fill [color=black] (10,2) circle (1.5pt);
\fill [color=black] (10.7,3) circle (1.5pt);

\fill [color=black] (10.7,2) circle (1.5pt);
\fill [color=black] (11.4,2) circle (1.5pt);
\fill [color=black] (12.4,3) circle (1.5pt);
\draw (6,0.5) node[anchor=north east]{$G_1$};

\fill [color=black] (12.4,1) circle (1.5pt);
\draw (11,0.5) node[anchor=north east]{$G_2$};

\end{tikzpicture}
\caption{}\label{G}
\end{figure}
\end{Example}

Combining Proposition~\ref{rank} with Proposition~\ref{two},  we obtain the following result.

 \begin{Corollary} \label{twoc} Under the same assumptions on $G$ as in Proposition~\ref{two},  the socle module $\Soc(I(G))$ is identified with an ideal of $\mathcal{F}(I)$.
 \end{Corollary}

We now come to the main result of this section.
If $G$ is a
unicyclic graph with the  unique odd cycle $C_{2k+1}$, we use the notion $E^*(G)$ to stand for the following subset of $E(G)$:
\[
E^{*}(G)=\{ e\in E(G)\setminus E(C_{2k+1})\:\;
 \text{ $e$ is not a leaf of $G$}  \}.
 \]
Here, an edge $e=y_1y_2\in E(G)$ is called a {\em leaf} of $G$ if either $y_1$ or $y_2$ is adjacent to only one vertex of $G$.

 \begin{Proposition}
 \label{u} Let $G$ be a
unicyclic graph containing an odd cycle $C_{2k+1}$ with
$\text{V}(C_{2k+1})=\{x_1,x_2,\ldots, x_{2k+1}\}$.  Let  $u\in S$ a monomial and $m$ a positive integer.  Then
$$\mathbf{u }\in (I^m:\mm)\setminus {I^m}\, \text{ if\,
and\, only\, if}\,\, \mathbf{u}=(x_1x_2\cdots x_{2k+1}) \prod_{e\in E_1}e^{m_e},$$ where $m_e\geq 1$ for all $e\in E_1$, $E^*(G)\subseteq E_1$,  and $\sum_{e\in E_1}m_e=m-k-1$.
\end{Proposition}

\begin{proof}  Sufficiecy.

Note that $\deg(\mathbf{u})=2k+1+2(m-k-1)=2m-1$, it is clear
that $\mathbf{u}\not\in I^m$.  We only need to show $y\mathbf{u}\in I^m$ for any $y\in V(G)$. Let $y\in V(G)$.

Case (i):\, $\dist(y,V(C_{2k+1}))=0$,  that is, $y\in V(C_{2k+1})$.
We may harmlessly assume that $y=x_1$. Then $$\mathbf{u} y=
(x_1x_2\cdots x_{2k+1}) (\prod_{e\in E_1}e^{m_e})
x_1= [(x_1x_2)(x_3x_4) \cdots (x_{2k+1}x_1))]\cdot \prod_{e\in
E_1}e^{m_e} \in I^m.$$

Case (ii):\, $\dist(y,C_{2k+1})>0$. Then
there exits  $x \in V(C_{2k+1})$, say $x=x_{1}$, such that $\dist(y,V(C_{2k+1}))=\dist(x,y)$.
Write the unique path between $y$ and $x_1$ as:
$$y=y_0-y_1-y_2-\cdots-y_{\ell}-x_1.$$ Suppose that $\ell =2k$ for some $k\geq 1$. Then, since $y_1y_2, y_3y_4,\ldots, y_{2k-1}y_{2k}$ are pairwise different edges   belonging to $E^*(G)$, we may write $y\mathbf{u}$ as: $$y\mathbf{u}= (x_1\cdots x_{2k+1})y_0(y_1y_2)\cdots (y_{2k-1}y_{2k})\mathbf{v},$$ where $\mathbf{v}$ is the product of  some edges. It follows that $$y\mathbf{u}=(y_0y_1)\cdots (y_{2k-2}y_{2k-1}) (y_{2k}x_1)(x_2x_3)\cdots (x_{2k}x_{2k+1})\mathbf{v}$$ and $y\mathbf{u}$ is also the product of  edges.  Since $\deg (y\mathbf{u})=2m$,  we have $y\mathbf{u}\in I^m$.
Suppose next that $\ell=2k+1$. Then we may write $y\mathbf{u}$ as $$y\mathbf{u}=(x_1\cdots x_{2k+1})y_0(y_1y_2)\cdots (y_{2k-1}y_{2k})(y_{2k+1}x_1)\mathbf{v}$$ such that $\mathbf{v}$ is the product of edges. Thus, $$y\mathbf{u}=(y_0y_1)\cdots (y_{2k}y_{2k+1})(x_1x_2)\cdots (x_{2k-1}x_{2k})(x_1x_{2k+1})\mathbf{v}$$ and it is also the product of edges. Note that $\deg(y\mathbf{u})=2m$, we have $y\mathbf{u}\in I^m$.

Necessity.

Let $\mathbf{u} \in (I^m:\mm)\setminus {I^m}$ be a monomial.  Then, according to Proposition~\ref{two}, there exists a subset $E_1\subseteq E(G)$ such that  $$\mathbf{u}=(x_1\cdots x_{2k+1})\prod_{e\in E_1}e^{m_e},$$ where $m_e\geq 1$ for any $e\in E_1$ and $\sum_{e\in E_1}m_e=m-k-1$. We only need to prove $E^*(G)\subseteq E_1$.

Let $e=y_1y_2\in E^*(G)$. There exits a unique vertex in $C_{2k+1}$, say  $x_1$, such that $\dist(x_1,y_i)=\dist(V(C_{2k+1}),y_i)$ for $i=1,2$. Moreover, we may harmlessly assume that $\dist(y_2,x_1)=\dist(y_1,x_1)+1$. It is clear that $y_2\notin \{x_1,\ldots,x_{2k+1}\}$. Since $y_1y_2$ is not a leaf, there exists a vertex $z$ of $G$ such that $\dist(z,x_1)=\dist(y_2,x_1)+1$ and $zy_2\in E(G)$.

 Since $z\mathbf{u}=x_1z(x_2x_3)\cdots (x_{2k}x_{2k+1})\prod_{e\in E_1}e^{m_e}\in I^m$,  it follows from Lemma~\ref{connectedness} that there exists an odd walk connecting $z$ to $x_1$:
\begin{equation}
\label{walk} \tag{W}
z=z_1-z_2-z_3-\cdots-z_{2t}-z_{2t+1}-z_{2t+2}=x_1
\end{equation}
such that
\begin{equation} \label{multiset} \tag{I}
\{\{ z_2z_3,\ldots,z_{2t}z_{2t+1}\}\}\subseteq \{\{m_ee: e\in E_1, x_2x_3, \ldots, x_{2k}x_{2k+1}\}\},
 \end{equation} where $m_ee$ means that $e$ appears $m_e$ times.  We may assume that $W$ has the shortest  length among such walks.

Let $\ell$ be the minimal of numbers $\{1,2,\ldots, 2t+2\}$ such that $z_{\ell}=x_1$. Then, any subset of $\{z_1,\ldots, z_{\ell-1}\}$ cannot form  any odd closed walk in $G$, since $\{x_1,\ldots,x_{2k-1}\}\cap \{z_1,\ldots, z_{\ell-1}\}=\emptyset$.
We claim that $z\notin \{z_2,\ldots, z_{\ell-1}\}$. It is clear that $z\neq z_2$. If $z=z_{2s}$ for some $s$ with $4\leq 2s\leq \ell-1$, then
$z=z_1-z_2-\cdots-z_{2s}=z$ is an odd closed walk, which is impossible. If $z=z_{2s+1}$ for some $s>0$ with $2s+1\leq \ell-1$, then $z=z_{2s+1}-z_{2s+2}-\cdots-z_{2t}-z_{2t+1}-z_{2t+2}=x_1$ is a shorter walk than (\ref{walk}), a contradiction.
Thus,  $z\notin \{z_2,\ldots, z_{\ell-1}\}$, as claimed. From this it follows that $\dist(z_2,x_1)=\dist(z_1,x_1)-1$  and $z_2=y_2$. In a similar way, we have $y_2\notin \{z_3,\ldots, z_{\ell-1}\}$ and  $z_3=y_1$.  Hence,   $e\in E_1$ by (\ref{multiset}).
\end{proof}

In the following, if $G$ is a unicyclic graph with odd cycle $x_1-x_2-\cdots-x_{2k+1}-x_1$,  we use $d_G$ to denote the number $$d_G=|E^*(G)|+k+1$$ and use $\mathbf{u}_G$ to denote the monomial $$\mathbf{u}_G=(x_1\cdots x_{2k+1})\prod_{e\in E^*(G)}e.$$  It is clear from Proposition~\ref{u} that $\mathbf{u}_G+I^{d_G}$ is the unique minimal generator of $\Soc(I(G))$.

 \begin{Theorem}
 \label{fiber} Let $G$ be a unicyclic nonbipartite graph. Then $\Soc(I(G))$ is  a free $\mathcal{F}(I(G))$-module of rank 1. More precisely,
$$\Soc(I(G))\cong \mathcal{F}(I(G))(-d_G+1).$$
\end{Theorem}

 \begin{proof} Denote $I(G)$ by $I$.  In view of  Proposition~\ref{u}, we have $$\Soc(I)=(\mathbf{u}_G+I^{d_G})\mathcal{F}(I)$$ and so it is a cyclic $\mathcal{F}(I)$ -module.  According to Proposition~\ref{twoc}, $\mathbf{u}_G+I^{d_G}$ has no non-zero annihilator in $\mathcal{F}(I)$, namely, $f(\mathbf{u}_G+I^{d_G})\neq 0$ for any $0\neq f\in \mathcal{F}(I)$.
Thus, $\Soc(I)$ is  a free $\mathcal{F}(I)$-module of rank 1. The last isomorphism follows since   $\mathbf{u}_G+I^{d_G}\in \Soc(I)$ is of degree $d_G-1$.
 \end{proof}

The converse of Theorem~\ref{fiber} is not true as shown by the following example.

\begin{Example}  \em Let $G$ be the graph obtained by adding an edge $x_1x_3$ to the 5-cycle $x_1-x_2-\cdots-x_5-x_1$. Then $x_1x_2x_3+I(G)^2$  is the unique minimal generator of $\Soc(I(G))$ and so $\Soc(I(G))$ is also a free  $\mathcal{F}(I(G))$-module of rank 1.
\end{Example}

In the final part of this section, we give an application of our results to the depth stability index of an edge ideal. Let $I$ be an ideal of $S$. Recall that the stability index of $I$, denoted by $\mbox{dsatb}(I)$, is defined to be $$\mbox{dsatb}(I)=\min\{t\geq 0:\depth(I^t)=\depth(I^{t+i}) \mbox{\ for all \ } i\geq 0 \}.$$

Let $G$ be any simple connected graph. By a {\it spanning subgraph} of $G$, we mean a connected  subgraph  of $G$ with the same vertex set as $G$.
Thus, a {\it spanning unicyclic nonbipartite subgraph} of $G$ means a spanning subgraph of $G$ which is  unicyclic and nonbipartite.
Assume now  $H$ is a spanning unicyclic nonbipartite subgraph of $G$. It is easy to check that $\mathbf{u}_H\in (I(G)^{d_H}:\mm)\setminus I(G)^{d_H}$ and  $\mathbf{u}_H+I(G)^{d_H}$ is a nonzero homogeneous element of degree $d_H-1$ in $\Soc(I(G))$. Note that $\mathrm{dstab}(I(G))=\min\{i+1: \Soc(I(G))_i\neq 0\}$, the following result follows.

\begin{Corollary} \label{dstab}Let $G$ be a simple connected nonbipartite graph. Then $$\mathrm{dstab}(I(G))\leq \min\{d_H: H \mbox{\ is a spanning unicyclic nonbipartite subgraph of\ }G \}.$$
\end{Corollary}

  Let $H$ be  a spanning unicyclic nonbipartite subgraph of $G$ such that the unique cycle of $H$ has  length $2k-1$, where $2k-1$ be the maximal length of odd cycles of $G$. Then $$d_H=|V(H)|-\epsilon_0(H)-k+1\leq |V(G)|-\epsilon_0(G)-k+1,$$ where $\epsilon_0(\bullet)$ is the number of leaves of a simple  graph $\bullet$.
From this fact, we see that  Corollary~\ref{dstab} implies \cite[Proposition 2.4]{T}, which says $\mathrm{dstab}(I(G))\leq |V(G)|-\epsilon_0(G)-k+1,$ if $G$ is a connected nonbipartite graph.

\begin{Example} \em Let $G$ be the complete graph with $2k-1$ vertices, where $k\geq 2$. According to Corollary~\ref{dstab}, we see that $\mathrm{dstab}(I(G))\leq 2$ and so $\mathrm{dstab}(I(G))=2$. However, by \cite[Proposition 2.4]{T}, we can only obtain $\mathrm{dstab}(I(G))\leq k$.
\end{Example}

The following example illustrates the inequality in Corollary~\ref{dstab} may be strict.
\begin{Example}\em
Let $G$ be the graph depict in Figure~\ref{K} and $I$ the edge ideal of $G$. Then $\Soc(I)$ is minimally generated by $x_1x_2x_3x_4x_5+I^3$ and  $x_1x_2x_3^3x_4x_5+I^4$ by  CoCoA \cite{A-B}. This implies  $\mbox{dstab} (I)=3$. Let $H_1=G\setminus \{x_1x_2\}$, $H_2=G\setminus \{x_4x_5\}$,
$H_3=G\setminus \{x_3x_1\}$, $H_4=G\setminus \{x_3x_2\}$, $H_5=G\setminus \{x_3x_4\}$, and $H_6=G\setminus \{x_3x_5\}$. Then $H_1,\ldots, H_6$ are all the  spanning unicyclic nonbipartite subgraphs of $G$. However, $d_{H_i}=4$ for $i=1,\ldots,6$.
\begin{figure}[ht!]

\begin{tikzpicture}[line cap=round,line join=round,>=triangle 45,x=1.5cm,y=1.5cm]

\draw (4.7,1)--(4.7,3)--(5.7,2)--(6.7,3); \draw (6.7,3)--(6.7,1)--(5.7,2); \draw (4.7,1)--(5.7,2);
\draw (4.7,1)--(3.7,1); \draw (4.7,3)--(3.7,3); \draw (6.7,1)--(7.7,1); \draw (6.7,3)--(7.7,3);

\draw (4.7,1) node[anchor=north east]{1};
\draw (5.6,2.2) node[anchor=north east]{3};
\draw (5,3.2) node[anchor=north east]{2};
\draw (3.7,3.2) node[anchor=north east]{6};
\draw (6.9,1) node[anchor=north east]{5};
\draw (3.7,1) node[anchor=north east]{7};

\draw (6.7,3.2) node[anchor=north east]{4};
\draw (8.1,3.2) node[anchor=north east]{8};

\fill [color=black] (6.7,3) circle (1.5pt);
\fill [color=black] (6.7,1) circle (1.5pt);
\fill [color=black] (5.7,2) circle (1.5pt);
\draw (8.1,1) node[anchor=north east]{9};

\fill [color=black] (4.7,3) circle (1.5pt);
\fill [color=black] (4.7,1) circle (1.5pt);

\fill [color=black] (3.7,1) circle (1.5pt);
\fill [color=black] (3.7,3) circle (1.5pt);

\fill [color=black] (7.7,3) circle (1.5pt);
\fill [color=black] (7.7,1) circle (1.5pt);

\end{tikzpicture}
\caption{}\label{K}
\end{figure}

\end{Example}

\section{Polymatroidal Ideals}

In this section we want to  investigate the socle modules $\Soc^*(I)$ of a polymatroidal ideal.
We refer to \cite{H-H-1} or \cite[Chapter 12]{H-H-b} for the definition and basic properties of polymatroidal ideals.

\medskip

 Let $I$ be  a polymatroidal ideal generated in degree $d$, then $I$ has a linear resolution. Recall that $\mathrm{soc}(I)$ is  the monomial ideal generated in degree $d-1$ such that $(I: \mm)=I+\mathrm{soc}(I)$. The first basic question in our context is
 \medskip

 {\em ``is $soc(I)$ always a polymatroial ideal if $I$ is a polymatroidal ideal?''}
\medskip

 We cannot answer this question in its full generality. In what  follows we confine our research to a special type of polymatroidal ideals--- PLP-polymatroidal ideals,  which were introduced firstly in \cite{S} and were redefined and studied further in \cite{L}.

Denote by $G(I)$ the unique minimal set of generators of a monomial ideal $I$.
Let $\mathbf{a}=(a_1,\ldots,a_n), \mathbf{b}=(b_1,\ldots,b_n),\bm{\alpha}=(\alpha_1,\ldots, \alpha_n),\bm{\beta}=(\beta_1,\ldots,\beta_n)$ be vectors in $\mathbb{Z}_+^n$. Assume that  $\alpha_1\leq \alpha_2\leq \cdots\leq \alpha_n=d$, $\beta_1\leq \beta_2\leq \cdots \leq \beta_n=d$ and $a_i\leq b_i$ and $\alpha_i\leq \beta_i$ for $i=1,\ldots,n$. Here, $d$ is a given integer. We also assume that $\alpha_1=a_1$ and $\beta_1=b_1$. Recall from \cite{L} that the  {\em  PLP-polymatroidal ideal}  $I$ of type $(\mathbf{a},\mathbf{b}|\bm{\alpha},\bm{\beta})$ is the polymatroidal ideal in $K[x_1,\ldots,x_n]$ generated by $\mathbf{x}^{\mathbf{u}}$  with $\mathbf{u}=(u_1,\ldots,u_n)$ satisfying
\begin{equation}
\begin{split}
\label{definition}
a_i\leq u_i\leq b_i, \mbox{\quad for\quad   } i=1,\ldots,n,\\
\alpha_i\leq u_1+\cdots+ u_i\leq \beta_i, \mbox{\quad for\quad   } i=1,\ldots,n.
\end{split}
\end{equation}

\medskip
The PLP-polymatoidal ideals generalizes ideals of Veronese type greatly.  Two special types of PLP-polymatroidal ideals  were studied in \cite{L} and it was shown among others that $\depth (S/I^k)$ and $\mathrm{Ass}(S/I^k)$ become constant at the same $k$ for such ideals $I$.
\medskip

 We call the PLP-polymatroidal ideal of type $(\mathbf{a},\mathbf{b}|\bm{\alpha},\bm{\beta})$ to be {\it basic} if $\mathbf{a}=0$. Since every PLP-polymatroidal ideal is the product of  a basic  PLP-polymatroidal ideal and  a monomial, we will only consider basic  PLP-polymatroidal ideals hereafter.

\begin{Lemma} \label{solution} Assume that $\mathbf{b}, \bm{\alpha},\bm{\beta}$ satisfy the conditions stated before and assume further that $\mathbf{a}=\mathbf{0}$.  The system (\ref{definition}) of inequalities has an integral solution if and only if $\beta_i+b_{i+1}+\cdots+b_j\geq \alpha_j$ for any pair $i,j$ with $1\leq i\leq j\leq n$.
\end{Lemma}

\begin{proof} Let $\mathbf{u}=(u_1,\ldots,u_n)$ be an integral solution of (\ref{definition}). Then for any $i,j$ with $1\leq i\leq j\leq n$, we have $\alpha_j\leq (u_1+\cdots+u_i)+u_{i+1}+\cdots+u_j\leq \beta_i+b_{j+1}+\cdots+b_n$. This proves the necessity.

Conversely, assume that   $\beta_i+b_{i+1}+\cdots+b_j\geq \alpha_j$ for any $i,j$ with $1\leq i\leq j\leq n$.  Then we set $u_1=\beta_1$, and $$u_{j+1}=\min\{\beta_i+b_{i+1}+\cdots+b_{j+1}| i\leq j+1 \}-(u_1+u_2+\cdots +u_j)$$ for $j=1,\ldots,n-1$. It is enough to show that $\mathbf{u}:=(u_1,\ldots,u_n)$ is an integral solution of (\ref{definition}).

Since $u_1+u_2+\cdots+u_j=\min\{\beta_i+b_{i+1}+\cdots+b_{j}| i\leq j\}$, it follows that $\alpha_j\leq u_1+u_2+\cdots+u_j\leq \beta_j$ for $j=1,\ldots, n$. In addition, we have  $$0\leq u_{j+1}=\min\{u_1+u_2+\cdots+u_j+b_{j+1}, \beta_{j+1}\}-u_1+u_2+\cdots+u_j\leq b_{j+1}$$
for $j=1,\ldots,n-1$. Here, to see $u_{j+1}\geq 0$, we need the assumption that $\beta_{j+1}\geq \beta_j$.  Hence, the vector $\mathbf{u}$ is indeed an integral solution of (\ref{definition}).
\end{proof}

\begin{Remark} \label{solutionr} \em In view of the proof of  Lemma~\ref{solution}, we see that it still holds if the condition that $\alpha_1\leq \alpha_2\leq \cdots \leq \alpha_n$ is dropped.
\end{Remark}

In what follows, we keep the assumptions on $\mathbf{b}=(b_1,\ldots,b_n), {\bm \alpha}=(\alpha_1,\ldots,\alpha_n), {\bm \beta}=(\beta_1,\ldots,\beta_n)$ stated before. In particular, $\alpha_n=\beta_n=d$.

\begin{Proposition} \label{soc1} Let $I$ be a basic PLP-polymatroidal ideal of type $$(\mathbf{0}, (b_1,\ldots,b_n)|(\alpha_1,\ldots,\alpha_n),(\beta_1,\ldots,\beta_n)).$$
Then $\mathrm{soc}(I)$ is a basic PLP-polymatroidal ideal of type \begin{equation} \tag{*}\label{soc}
(\mathbf{0}, (b_1-1,\ldots,b_n-1)|(\alpha_1,\ldots,\alpha_{n-1},\alpha_n-1),(\beta_1-1,\ldots,\beta_n-1)).
\end{equation}
\end{Proposition}
\begin{proof} We use $\mathbf{u}\in (\ref{definition})$ to denote the condition that  $\mathbf{u}$ is an integral solution of the system (\ref{definition}) for any $\mathbf{u}\in \mathbb{Z}_+^n$. Note that $a_1=\cdots=a_n=0$ in (\ref{definition}) by our assumption.

Denote by $J$ the PLP-polymatroidal ideal of type (\ref{soc}). We need to show $\mathrm{soc}(I)=J$. Let $\mathbf{x^v}$ be a minimal generator of $\mathrm{soc}(I)$. Then $v_1+\cdots+v_n=d-1$ and   $\mathbf{v+e_i}\in (\ref{definition})$ for any $i\in [n]$, where $\mathbf{e_1},\ldots,\mathbf{e_n}$ is the canonical basis of $\mathbb{Z}_+^n$.   Fix  $2\leq i\leq n-1$.  Since  $\mathbf{v+e_n}\in (\ref{definition})$, we have $v_1+\cdots+v_i=(\mathbf{v+e_n})([i])\geq \alpha_i$. Here, for a vector $\mathbf{u}=(u_1,\ldots,u_n)\in \mathbb{Z}_+^n$, $\mathbf{u}([i])$ means the number $u_1+\cdots+u_i$.   On the other hand, we have $v_1+\cdots+v_i=(\mathbf{v+e_i})([i])-1\leq \beta_i-1$, since $\mathbf{v+e_i}\in (\ref{definition})$.  Thus, $\alpha_i\leq v_1+\cdots+v_i\leq \beta_i-1$ for $i=2,\ldots,n-1$. It is also clear  that $0\leq v_i\leq b_i-1$ for $i=1,\ldots,n$. This implies $\mathrm{soc}(I)\subseteq J$.

  Conversely, if $\mathbf{x}^{\mathbf{v}}$ is a minimal generator of $J$,  then it is routine to check that $\mathbf{v+e_i}\in (\ref{definition})$ for all $i=1,\ldots,n$. Hence $J\subseteq \mathrm{soc}(I)$, as desired.\end{proof}

\begin{Corollary} \label{depth}  Let $I$ be a basic PLP-polymatroidal ideal of type $$(\mathbf{0}, (b_1,\ldots,b_n)|(\alpha_1,\ldots,\alpha_n),(\beta_1,\ldots,\beta_n)).$$ Then
$\depth (S/I)=0$ if and only if the following inequalities hold:
\begin{flalign*}
& \qquad \qquad b_i\geq 1   \mbox{ \quad for all }  i=1,\ldots,n; &
\\    & \qquad \qquad   \alpha_i\leq \beta_i-1 \mbox{\quad for \quad } i=2,\ldots,n-1;&
\\&  \qquad \qquad  \beta_i+b_{i+1}+\cdots+b_j\geq \alpha_j+j-i+1 \mbox{ \quad for all } 1\leq i\leq j\leq n-1; &
\\    &  \qquad \qquad  \beta_i+b_{i+1}+\cdots+b_n\geq d+n-i \mbox{ \quad for all } 1\leq i\leq n.&
\end{flalign*}
\end{Corollary}

\begin{proof} Since $\depth (S/I)=0$ if and only if $\mathrm{soc}(I)\neq 0$, this result follows immediately by Proposition~\ref{soc1} together with Lemma~\ref{solution}.
\end{proof}

We now clarify when $\Soc(I)\neq 0$ in the case that $I$ is a basic PLP-polymatroidal ideal.
\begin{Corollary}  Let $I$ be a basic PLP-polymatroidal ideal of type $$(\mathbf{0}, (b_1,\ldots,b_n)|(\alpha_1,\ldots,\alpha_n),(\beta_1,\ldots,\beta_n)).$$ Then
the following statements are equivalent:

$\mathrm{(a)}$ $\Soc^*(I)\neq 0$;

$\mathrm{(b)}$ $S/I^m$ has limit depth zero, namely, $\lim_{m\rightarrow \infty}(S/I^m)=0$;

$\mathrm{(c)}$ $\ell (I)=n$;

$\mathrm{(d)}$  the following inequalities hold at the same time:
\begin{flalign*}
&\qquad \qquad  b_i\geq 1 \mbox{\quad for \quad } i=1,\ldots,n; &
\\ &\qquad \qquad \alpha_i\leq \beta_i-1 \mbox{\quad for \quad } i=2,\ldots,n-1;&\\
& \qquad \qquad \beta_i+b_{i+1}+\cdots+b_j> \alpha_j \mbox{\quad  for \quad} 1\leq i< j\leq n.&
\end{flalign*}
\end{Corollary}

\begin{proof} $(a)\Leftrightarrow (b)$ is obvious and $(b)\Leftrightarrow (c)$ follows from \cite[Corollary 2.5]{H-R-V}.

 $(b)\Leftrightarrow (d)$   According to  \cite[Proposition 2.10]{L}, $I^m$ is a basic PLP-polymatroidal ideal of type $(\mathbf{0},m\mathbf{b}|m{\bm \alpha}, m{\bm \beta})$. From this as well as Corollary~\ref{depth}, the result follows.\end{proof}

\begin{Example}\label{3.7} \em  Let $t\geq 2$ and let $I=(x^t, xy^{t-2}z, y^{t-1}z)\subseteq S=K[x,y,z]$. Then $\depth (S/I^n)$ is 1  if $n\leq t-1$ and is 0 if $n\geq t$, according to \cite[Proposition 2.5]{M-S-T}. This implies $\Soc^*(I)$ is generated in degrees $\geq t-1$.
\end{Example}

Example~\ref{3.7} demonstrates that   $\Soc^*(I)$ may be  generated in degrees far larger than $n$.    But in the case that $I$ is a PLP-polymatroidal ideal, the situation is different.

\begin{Theorem}\label{mainPLP} Let $I$ be a PLP-polymatroidal ideal of type $$(\mathbf{0}, (b_1,\ldots,b_n)|(\alpha_1,\ldots,\alpha_n),(\beta_1,\ldots,\beta_n)).$$ Then $\Soc^*(I)$ and $\Soc(I)$ are generated in degrees $< n$.
\end{Theorem}
\begin{proof} Since $\Soc(I)\cong \Soc^*(I)/\mm \Soc^*(I)$, we only need to prove  $\Soc^*(I)$  are generated in degrees $< n$.   By \cite[Proposition 2.10]{L} and Proposition~\ref{soc1}, we see that $\mathrm{soc}(I^{\ell})$ is of type $$(\mathbf{0}, ({\ell}b_1-1,\ldots,{\ell}b_n-1)|({\ell}\alpha_1,\ldots,{\ell}\alpha_{n-1}, {\ell}\alpha_n-1),({\ell}\beta_1-1,\ldots,{\ell}\beta_n-1)).$$  It is enough to show that $\mathrm{soc}(I^{k+1})\subseteq I\mathrm{soc}(I^k)$ for all $k\geq n-1$.   Let $\mathbf{u}=(u_1,\ldots,u_n)$ be a vector in $\mathbb{Z}_+^n$ such that $\mathbf{x}^{\mathbf{u}}$ is a minimal generator of $\mathrm{soc}(I^{k+1})$. Then we write $u_i=(k+1)s_i+t_i$ with $0\leq t_i\leq k$  for every $i\in [n]$.  We want to find a suitable vector $(\lambda_1,\ldots,\lambda_n)\in \mathbb{Z}_+^n$ such that $\mathbf{x^w}$ and $\mathbf{x^v}$ belong to $G(I)$ and $ G(\mathrm{soc}(I^k))$ respectively, where  $\mathbf{w}:=(s_1+\lambda_1,\ldots,s_n+\lambda_n)$ and $\mathbf{v}:=(ks_1+t_1-\lambda_1,\ldots, ks_n+t_n-\lambda_n)$. To this end, we may write $\mathbf{w}:=(s_1+\lambda_1,\ldots,s_n+\lambda_n)$ and $\mathbf{v}:=(ks_1+t_1-\lambda_1,\ldots, ks_n+t_n-\lambda_n)$ with $(\lambda_1,\ldots,\lambda_n)\in \mathbb{Z}_+^n$ undetermined.

Since  $(k+1)s_i+t_i\leq (k+1)b_i-1$, we have $s_i\leq b_i-\frac{1+t_i}{k+1}$ and so \begin{equation} \label{L}
\begin{split}
w_i\leq b_i-\frac{1+t_i}{k+1}+\lambda_i \mbox{\quad for \quad} i=1,\cdots,n\\
v_i\leq kb_i-\frac{k(1+t_i)}{k+1}+t_i-\lambda_i  \mbox{\quad for \quad} i=1,\cdots,n
\end{split}
\end{equation}

Set $\Lambda_i=\lambda_1+\cdots+\lambda_i$ for $i=1,\ldots,n$. For $i=2,\ldots,n-1$, since $$(k+1)\alpha_i\leq (k+1)(s_1+\cdots+s_i)+t_1+\cdots+t_i\leq (k+1)\beta_i-1,$$
it follows that $$\alpha_i-\frac{t_1+\cdots+t_i}{k+1}\leq s_1+\cdots+s_i\leq \beta_i-\frac{1+t_1+\cdots+t_i}{k+1}.$$
Hence, for $i=2,\ldots,n-1$, we have
\begin{equation} \label{S}
\begin{split} \alpha_i-\frac{t_1+\cdots+t_i}{k+1}+\Lambda_i \leq w_1+\cdots+w_i\leq \beta_i-\frac{1+t_1+\cdots+t_i}{k+1}+\Lambda_i,\\
k\alpha_i+\frac{t_1+\cdots+t_i}{k+1}-\Lambda_i\leq v_1+\cdots+v_i\leq k\beta_i+\frac{t_1+\cdots+t_i-k}{k+1}-\Lambda_i.
\end{split}
\end{equation}

From the inequalities (\ref{L}) and (\ref{S}),  it is not difficult to see that  if
\begin{flalign}
\label{5} & \qquad \qquad \Lambda_n=\frac{t_1+\cdots+t_n+1}{k+1}=\beta_n-s_1-\cdots-s_n; &
\\ \label{6}   & \qquad \qquad   \frac{t_1+\cdots+t_i}{k+1}-\frac{k}{k+1}<\Lambda_i<1+\frac{t_1+\cdots+t_i}{k+1}, \mbox{\quad} i=2,\ldots,n-1;&
\\ \label{7}   &  \qquad \qquad  \frac{t_i-k}{k+1}<\lambda_i<1+\frac{1+t_i}{k+1}, \mbox{\quad} i=1,\ldots,n. &
\end{flalign}
then  $\mathbf{x^w}\in G(I)$ and $\mathbf{x^v}\in G(\mathrm{soc}(I^k))$.
In what follows, we show that such $\lambda_i$'s exist.

Note that the inequalities in (\ref{7}) are equivalent to $$\lambda_i=1 \mbox{\quad if \quad} t_i=k  \mbox{ \qquad  and \qquad  }0\leq \lambda_i\leq 1 \mbox{\quad if \quad } 0\leq t_i\leq k-1.$$
Let $[a]$ denote the largest integer $<a$ and $\lfloor a\rfloor$ the largest integer $\leq a$. Under  these notations, the above system of inequalities (\ref{5}), (\ref{6}), (\ref{7})  is equivalent to the following system:
\begin{equation}\label{8} \left\{
                   \begin{array}{ll}
                     \lambda_i=1 \mbox{\qquad if \quad} t_i=k  \mbox{ \qquad  and \qquad  }0\leq \lambda_i\leq 1 \mbox{\qquad if \quad } 0\leq t_i\leq k-1; \\
     \\                \lfloor\frac{t_1+\cdots+t_i}{k+1}-\frac{k}{k+1}\rfloor\leq \lambda_1+\cdots+\lambda_i\leq 1+[\frac{t_1+\cdots+t_i}{k+1}] \mbox{\quad} i=1,2,\ldots,n-1; \\
\\
                     \lambda_1+\cdots+\lambda_n=\frac{t_1+\cdots+t_n+1}{k+1}.
                   \end{array}
                 \right.
\end{equation}

\medskip
Set $\epsilon_i=\lambda_i-1$ if $t_i=k$ and $\epsilon_i=\lambda_i$ otherwise. Let $c_i$ denote the number $|\{1\leq j\leq i| t_j=k\}|.$
Then the system (\ref{8}) can be rewritten as:
\begin{equation}
\label{system}
\left\{
  \begin{array}{ll}
    \epsilon_i=0 \mbox{\qquad if \quad} t_i=k  \mbox{ \qquad  and \qquad  }0\leq \epsilon_i\leq 1 \mbox{\qquad if \quad} 0\leq t_i\leq k-1; \\
    \\
   \max\{\lfloor\frac{t_1+\cdots+t_i}{k+1}-\frac{k}{k+1}\rfloor-c_i,0\}\leq \epsilon_1+\cdots+\epsilon_i\leq 1+[\frac{t_1+\cdots+t_i}{k+1}]-c_i; \\
    \\
    \epsilon_1+\cdots+\epsilon_n=\frac{t_1+\cdots+t_n+1}{k+1}-c_n.
  \end{array}
\right.
\end{equation}

Since $k\geq n-1\geq c_n-1$, it follows that $\frac{t_1+\cdots+t_n+1}{k+1}-c_n\geq \frac{c_nk+1}{k+1}-c_n=\frac{1-c_n}{k+1}>-1$ and so $$\frac{t_1+\cdots+t_n+1}{k+1}-c_n\geq 0,$$  because $\frac{t_1+\cdots+t_n+1}{k+1}-c_n$ is an integer. Similarly, we have $$1+[\frac{t_1+\cdots+t_i}{k+1}]-c_i\geq 0 \mbox{\quad for \quad} i=1,\ldots,n-1.$$

For $i=1,\cdots,n$, we set $$h_i=\min\{\frac{t_1+\cdots+t_n+1}{k+1}-c_n, 1+[\frac{t_1+\cdots+t_j}{k+1}]-c_j, j=i,\ldots,n-1\}.$$  Also we denote  $g_n=\frac{t_1+\cdots+t_n+1}{k+1}-c_n$, and $g_i=\lfloor\frac{t_1+\cdots+t_i}{k+1}-\frac{k}{k+1}\rfloor-c_i$  for $i=1,\ldots,n-1.$
Then $h_n\geq h_{n-1}\geq \cdots\geq h_2\geq h_1\geq 0$. We claim that  $h_i\geq g_i$ for $i=1,\ldots,n$. For this, it is enough to prove that
 \begin{equation} \label{e} 1+[\frac{t_1+\cdots+t_j}{k+1}]-c_j\geq \lfloor\frac{t_1+\cdots+t_i}{k+1}-\frac{k}{k+1}\rfloor-c_i
 \end{equation} for  $1\leq i\leq j\leq n$.
 Let $s=c_j-c_i$. Then, since $k\geq n-1\geq s$, we have $$1+[\frac{t_1+\cdots+t_j}{k+1}]\geq \frac{t_1+\cdots+t_i+sk}{k+1}$$$$\geq  \frac{t_1+\cdots+t_i+sk+s-k}{k+1}\geq \frac{t_1+\cdots+t_i-k}{k+1}+s.$$
 This proves the desired inequality (\ref{e}) and so our claim follows.

 Set $e_i=0$ if $t_i=k$ and $e_i=1$ otherwise.  Note that $g_1=0$ and $e_1=h_1$. Hence, the system (\ref{system}) can be rewritten as follows.
\begin{equation}
\label{system1}
\left\{
  \begin{array}{ll}
    0\leq \epsilon_i \leq e_i,  \mbox{\qquad} i=1,\ldots,n; \\
    \\
    \max\{g_i,0\}\leq \epsilon_1+\cdots+\epsilon_i\leq h_i, \mbox{\qquad} i=2,\ldots,n.
  \end{array}
\right.
\end{equation}

In view of Lemma~\ref{solution} and Remark~\ref{solutionr}, to  prove that the system (\ref{system1}) has an integral solution, it is enough to show that
  $h_i+e_{i+1}+\cdots+e_j\geq g_j$ for all pairs
  $n\geq j\geq i\geq 1$. Fix $i,j$ with  $n\geq j\geq i\geq 1$.  If   $i=n$, there is nothing to prove.  If $i\leq n-1$ and $$h_i=\min\{\frac{t_1+\cdots+t_n+1}{k+1}-c_n,  1+[\frac{t_1+\cdots+t_{\ell}}{k+1}]-c_{\ell},\ell=j,\ldots,n-1\},$$  then $h_i=h_j\geq g_j$ and we are  done. So it remains to consider the case when $1\leq i< n$ and when $h_i=1+[\frac{t_1+\cdots+t_{\ell}}{k+1}]-c_{\ell}$ for some $\ell\in \{i,\ldots,j-1\}.$
   Note that $e_{i+1}+\cdots+e_j=j-i-c_j+c_i$. Hence it is enough to show that
  $$1+[\frac{t_1+\cdots+t_{\ell}}{k+1}]-c_{\ell}+j-i-(c_j-c_i)\geq \lfloor\frac{t_1+\cdots+t_j}{k+1}-\frac{k}{k+1}\rfloor-c_j.$$
  Since $c_{\ell}-c_i\leq \ell-i$, it is enough to prove
  $$1+[\frac{t_1+\cdots+t_{\ell}}{k+1}]+j-\ell\geq \lfloor\frac{t_1+\cdots+t_j}{k+1}-\frac{k}{k+1}\rfloor.$$
  But this is clear and so our proof is completed.
\end{proof}

A finitely generated graded module $M$ is called {\it equi-generated} if it generated in a single degree, namely, $M$ is generated in degree $d$, where $d=\min\{i\in \mathbb{Z}:  M_i\neq 0\}$. Finally, we discuss when $\Soc^*(I)$ is equi-generated when $I$ is an ideal of Veronese type.

\medskip
Let $I=I_{(a_1,\ldots,a_n),d}$ be an ideal of Veronese type (cf. \cite{H-R-V}) and let   $\rho$ be the rank function of the  polymatroid corresponding to $I$ (see \cite{H-H-1} or \cite{H-H-b} for the definition and basic properties of rank functions). Denote $(a_1,\ldots,a_n)$ by $\mathbf{a}$. Then we have the following observations.

\begin{enumerate}
  \item $\mathrm{depth}\ S/I=0$ if and only if $a_i\geq 1$ for all $i\in [n]$ and $\sum_{i\in [n]}a_i-d\geq n-1$;
\item $\Soc^*(I)\neq 0$  if and only if $a_i\geq 1$ for all $i\in [n]$ and $\sum_{i\in [n]}a_i>d$;
  \item $\rho(A)=\min\{\mathbf{a}(A), d\}$ for any $A\subseteq [n]$, here $\mathbf{a}(A)=\sum_{i\in A}a_i$;
  \item $I^k=I_{k\mathbf{a}, kd}$ for all $k\geq 1$;
\item $\mathrm{soc} (I)=I_{\mathbf{a-1},d-1}$. Here, $\mathbf{a-1}=(a_1-1,a_2-1,\ldots, a_n-1)$.
\end{enumerate}

\begin{Proposition}
\label{single}
Let $I=I_{(a_1,\ldots,a_n),d}=I_{\mathbf{a},d}$. Suppose that $a_i\geq 1$ for all $i\in [n]$ and $\mathbf{a}([n])>d$. Then the following statements are equivalent:

$(1)$ $\Soc(I)$ is equi-generated;

$(2)$ $\Soc^*(I)$ is equi-generated;

$(3)$  For any $A\subseteq [n]$ with $\mathbf{a}(A)>d$, one has $k_0(\mathbf{a}(A)-d)\geq |A|-1$. Here $k_0$ is the least number $k$ such that $k(\mathbf{a}([n])-d)\geq n-1$ and $|A|$ denotes the number of elements of $A$.
\end{Proposition}

\begin{proof}

$(1)\iff (2)$  This follows from Nakayama Lemma as well as the isomorphism $\Soc(I)\cong \Soc^*(I)/ \mm \Soc^*(I)$.

$(2)\Rightarrow (3)$  We only need to prove that if  there exits
$A\subseteq [n]$ such that  $\mathbf{a}(A)>d$    and  $k_0(\mathbf{a}(A)-d)< |A|-1$, then $I\mathrm{soc}(I^{k_0})\neq \mathrm{soc}(I^{k_0+1})$.

Let $\rho_1$ and $\rho_2$ be the rank functions of   $I\mathrm{soc}(I^{k_0})$ and $\mathrm{soc}(I^{k_0+1})$ respectively. Then $$\rho_1(A)=\min\{\mathbf{a}(A),d\}+\min\{k_0\mathbf{a}(A)-|A|, k_0d-1\}=d+k_0\mathbf{a}(A)-|A|.$$
Here, the second equality follows by the choice of the subset $A$.
On the other hand, $\rho_2(A)=\min\{(k_0+1)\mathbf{a}(A)-|A|,(k_0+1)d-1\}.$
We claim that $$d+k_0\mathbf{a}(A)-|A|\notin \{(k_0+1)\mathbf{a}(A)-|A|,(k_0+1)d-1\}.$$
In fact, if $d+k_0\mathbf{a}(A)-|A|=(k_0+1)\mathbf{a}(A)-|A|$ then $\mathbf{a}(A)=d$. This is contradicted to our assumption on $A$;
if $d+k_0\mathbf{a}(A)-|A|=(k_0+1)d-1$, then $k_0(\mathbf{a}(A)-d)=|A|-1$. This is also contradicted to our assumption on $A$. Thus, our claim has been proved.
From this claim it follows that $\rho_1(A)\neq \rho_2(A)$ and $I\mathrm{soc}(I^{k_0})\neq \mathrm{soc}(I^{k_0+1})$. Thus $\Soc^*(I)$ is not equi-generated.

$(3)\Rightarrow (2)$ Assume (3) holds. We will prove  $I\mathrm{soc}(I^{k_0})= \mathrm{soc}(I^{k_0+1})$ for all $k\geq k_0$. Fix $k\geq k_0$ and denote by $\rho_1$ and $\rho_2$ the rank functions of $I\mathrm{soc}(I^{k_0})$ and $\mathrm{soc}(I^{k_0+1})$  respectively. It is clear that $$\rho_1(A)=\min\{\mathbf{a}(A),d\}+\min\{k\mathbf{a}(A)-|A|, kd-1\}$$
and $$\rho_2(A)=\min\{(k+1)\mathbf{a}(A)-|A|,(k+1)d-1\}.$$
We now check that $\rho_1(A)=\rho_2(A)$ for all $\emptyset \neq A\subseteq [n]$.

Fix $\emptyset \neq A\subseteq [n]$. If $\mathbf{a}(A)\leq d$, then $\ell \mathbf{a}(A)-|A|\leq  \ell d-1$ for $\ell=k,k+1$, and so $\rho_1(A)=(k+1)\mathbf{a}(A)-|A|=\rho_2(A)$.

If $\mathbf{a}(A)> d$, then $\ell \mathbf{a}(A)-|A|\geq  \ell d-1$ for $\ell=k,k+1$ by (3). This  implies $\rho_1(A)=(k+1)d-1=\rho_2(A)$. Therefore, $\rho_1=\rho_2$ and $I\mathrm{soc}(I^{k_0})= \mathrm{soc}(I^{k_0+1})$ for all $k\geq k_0$, as desired. \end{proof}

In general $\Soc^*(I)$ is not equi-generated even if $I$ is an ideal of Veronese type.

\begin{Example} \em Let $I=I_{(3,3,1,2),6}$. Then $\mathrm{soc}(I)\neq 0$ and so $k_0=1$. Let $A=\{1,2,3\}$. Then $\mathbf{a}(A)=7>6=d$, but $k_0(\mathbf{a}(A)-6)<2=|A|-1$. Hence $\Soc^*(I)$ is not equi-generated by Proposition~\ref{single}.
\end{Example}

{\bf \noindent acknowledgement:}  The paper was written while the first author and the third author were visiting the Department of Mathematics of University
Duisburg–Essen Germany. They want to express their thanks for its hospitality. This research  is supported by  NSFC (No. 11971338).

\end{document}